\documentclass[12pt, reqno]{amsart}
\usepackage{amsmath, amsthm, amscd, amsfonts, amssymb, graphicx, color}
\usepackage[bookmarksnumbered, colorlinks, plainpages]{hyperref}
\hypersetup{colorlinks=true,linkcolor=red, anchorcolor=green, citecolor=cyan, urlcolor=red, filecolor=magenta, pdftoolbar=true}

\newtheorem{theorem}{Theorem}[section]
\newtheorem{lemma}[theorem]{Lemma}
\newtheorem{proposition}[theorem]{Proposition}
\newtheorem{corollary}[theorem]{Corollary}
\theoremstyle{definition}
\newtheorem{definition}[theorem]{Definition}

\theoremstyle{remark}
\newtheorem{remark}[theorem]{Remark}
\numberwithin{equation}{section}

\begin{document}

\setcounter{page}{1}

\title[Oscillatory integrals for Mittag-Leffler functions]{Oscillatory integrals for Mittag-Leffler functions with two variables}

\author[I.A.Ikromov, M.Ruzhansky, A.R.Safarov]{Isroil A. Ikromov, Michael Ruzhansky, Akbar R. Safarov$^{*}$}

\address{\textcolor[rgb]{0.00,0.00,0.84}{Isroil A. Ikromov \newline V.I. Romanovsky Institute of Mathematics of the Academy of Sciences of Uzbekistan \newline Olmazor district, University 46, Tashkent, Uzbekistan \newline Samarkand State University \newline Department of Mathematics \\
15 University Boulevard
 \newline  Samarkand, 140104, Uzbekistan }}
\email{\textcolor[rgb]{0.00,0.00,0.84}{ikromov1@rambler.ru}}
\address{\textcolor[rgb]{0.00,0.00,0.84}{Michael Ruzhansky\newline Department of Mathematics: Analysis, Logic and Discrete Mathematics
  \newline  Ghent University,  \newline  Krijgslaan 281, Ghent, Belgium,  \newline School of Mathematical Sciences, Queen Mary University of London,  \newline United Kingdom}}
\email{\textcolor[rgb]{0.00,0.00,0.84}{michael.ruzhansky@ugent.be}}
\address{\textcolor[rgb]{0.00,0.00,0.84}{Akbar R.Safarov \newline Uzbek-Finnish Pedagogical Institute \newline Spitamenshox 166, Samarkand, Uzbekistan \newline Samarkand State University \newline Department of Mathematics \\
15 University Boulevard
 \newline  Samarkand, 140104, Uzbekistan }}
\email{\textcolor[rgb]{0.00,0.00,0.84}{safarov-akbar@mail.ru}}

\thanks{All authors contributed equally to the writing of this paper. All authors read and approved the final manuscript.}

\let\thefootnote\relax\footnote{$^{*}$Corresponding author}



\begin{abstract}
In this paper  we  consider the problem of estimation of oscillatory integrals with Mittag-Leffler functions in two variables. The generalisation is that we replace the exponential function with the Mittag-Leffler-type function, to study oscillatory type integrals.

\end{abstract}
\maketitle
\tableofcontents
\section{Introduction}

The function $E_{\alpha}(z)$ is named after the Swedish mathematican G\"{o}sta Magnus Mittag-Leffler (1846-1927) who defined it by a power series
\begin{equation}\label{Formul3}
E_{\alpha}(z)=\sum_{k=0}^{\infty}\frac{z^{k}}{\Gamma(\alpha k+1)},\,\,\ \alpha\in\mathbb{C}, Re(\alpha)>0,
\end{equation}
and studied its properties in 1902-1905 in several subsequent notes \cite{ML1, ML2, ML3, ML4} in connection with his summation method for divergent series.

 A classical generalization of the Mittag-Leffler function, namely the two-parametric Mittag-Leffler function is
\begin{equation}\label{Formul4}
E_{\alpha,\beta}(z)=\sum_{k=0}^{\infty}\frac{z^{k}}{\Gamma(\alpha k+\beta)},\,\,\ \alpha,\beta\in\mathbb{C}, Re(\alpha)>0,
\end{equation}
which was deeply investigated independently by Humbert and Agarval in 1953 (\cite{Aga53,Hum53,HumAga53})  and by Dzherbashyan in 1954
(\cite{Dzh54a,Dzh54b,Dzh54c}) as well as in \cite{Rudolf}.

 It has the property that
\begin{equation}\label{Form1}
E_{1,1}(x)=e^{x}, \text{and we can refer to \cite{Pod} for other properties.}
\end{equation}

In harmonic analysis one of the most important estimates for oscillatory integral is van der Corput lemma  \cite{MichaelRuzhansky,MichaelRuzhansky2012,Ruzhansky,VanDer}.  Estimates for oscillatory integrals with polynomial phases can be found, for instance, in papers  \cite{AKC, mat sbor, Safarov, Safarov1, Safarov2}. In the current paper we replace the exponential function with the Mittag-Leffler-type function and study oscillatory type integrals \eqref{int1}. In the papers \cite{Ruzhansky} and \cite{Ruzhansky2021}  analogues of the van der Corput lemmas involving Mittag-Leffler functions for one dimensional integrals have been considered. We extend results of \cite{Ruzhansky} and \cite{Ruzhansky2021} for two-dimensional integrals with phase having some simple singularities. Analogous problem on estimates for Mittag-Leffler functions with the smooth phase functions of two variables having simple singularities  was considered in \cite{Ruzhansky2022} and \cite{Safarov3}.
\section{Preliminaries}

\begin{definition}
An oscillatory integral with phase $f$ and amplitude $a$ is an integral of the form
\begin{equation}\label{int11}
J(\lambda ,f,a)=\int _{{\bf {\mathbb R}}^{n} }a(x)e^{i\lambda f(x)} dx,
\end{equation}
where $a\in C_{0}^{\infty}({\bf {\mathbb R}}^{n})$ and $\lambda\in{\bf{\mathbb R}}$.
\end{definition}

If the support of $a$ lies in a sufficiently small neighborhood of the origin and $f$ is an analytic function at $x=0,$ then for $\lambda\rightarrow\infty$
the following  asymptotic expansion  holds (\cite{Karpushkin1980}):
 \begin{equation}\label{int22}
J(\lambda,f,a)\approx e^{i\lambda f(0)}\sum_{s}\sum_{k=0}^{n-1}b_{s,k}(a)\lambda^{s}(\ln\lambda)^{k},
\end{equation}
where $s$ belongs to a finite number of arithmetic progressions, independent of
$a,$ composed of negative rational numbers, $b_{s,k}$ is a distribution with support in the critical set $\{x:\nabla f(x)=0\}$.

Inspired by the terminology
from \cite{Arn}, we refer to the maximal value of $s$, denoting it by $\alpha$ in this case, as the \emph{growth index} of $f$, or the oscillation index at
the origin, and the corresponding value of $k$ is referred to as the \emph{multiplicity}.

More precisely, the \emph{multiplicity} of the oscillation index of an analytic phase at a critical point is
the maximal number $k$ possessing the property: for any neighbourhood of the
critical point there is an amplitude with support in this neighbourhood for which
in the asymptotic series \eqref{int22} the coefficient $b_{s,k}(a)$ is not equal to zero. The
multiplicity of the oscillation index will be denoted by $m$ (see \cite{Arn}).

Let $f$ be a smooth real-valued function defined on a neighborhood of the origin in $\mathbb{R}^{2}$ with $f(0,0)=0,$ $\nabla f(0,0)=0,$ and consider the associated Taylor series
$$f(x_{1},x_{2})\sim\sum_{j,k=0}^{\infty}c_{jk}x_{1}^{j}x_{2}^{k}$$
of $f$ centered at the origin. The set
$$\Im(f):=\{(j,k)\in\mathbb{N}^{2}:c_{jk}=\frac{1}{j!k!}\partial_{x_{1}}^{j}\partial_{x_{2}}^{k}f(0,0)\neq0\}$$
is called the Taylor support of $f$ at $(0,0).$ We shall always assume that
$$\Im(f)\neq\emptyset,$$
i.e., that the function $f$ is of finite type at the origin. If $f$ is real analytic, so that the Taylor series converges to $f$ near the origin, this just means that $f\neq0.$ The \emph{Newton polyhedron} $\aleph(f)$ of $f$ at the origin is defined to be the convex hull of the union of all the
quadrants $(j,k)+\mathbb{R}_{+}^{2},$ with $(j,k)\in\Im(f).$ The associated \emph{Newton diagram} $\aleph_{d}(f)$ in the sense of Varchenko \cite{Varch} is the union of all compact faces of the Newton polyhedron; here, by a \emph{face}, we mean an edge or a vertex.

We shall use coordinates $(t_{1},t_{2})$ for points in the plane containing the Newton polyhedron, in order to distinguish this plane from the $(x_{1},x_{2})$ - plane.

The \emph{distance} $d = d(f)$ between the Newton polyhedron and the origin in the sense of
Varchenko is given by the coordinate $d$ of the point $(d,d)$ at which the bisectrix $t_1 = t_2$
intersects the boundary of the Newton polyhedron.

The \emph{principal face} $\pi(f)$ of the Newton polyhedron of $f$ is the face of minimal dimension containing the point $(d,d).$ Deviating from the notation in \cite{Varch}, we shall call the series
$$f_{p}(x_{1},x_{2}):=\sum_{j,k\in\pi(f)}c_{jk}x_{1}^{j}x_{2}^{k}$$
the \emph{principal part} of $f$. In the case that $\pi(f)$ is compact, $f_{\pi}$ is a mixed homogeneous polynomial; otherwise, we shall consider $f_{\pi}$ as a formal power series.

Note that the distance between the Newton polyhedron and the origin depends on the
chosen local coordinate system in which $f$ is expressed. By a \emph{local analytic (respectively
smooth) coordinate system at the origin} we shall mean an analytic (respectively smooth)
coordinate system defined near the origin which preserves 0. If we work in the category
of smooth functions $f$, we shall always consider smooth coordinate systems, and if $f$ is
analytic, then one usually restricts oneself to analytic coordinate systems (even though
this will not really be necessary for the questions we are going to study, as we will see).
The \emph{height} of the analytic (respectively smooth) function $f$ is defined by
$$h:=h(f):=sup\{d_{x}\},$$
where the supremum is taken over all local analytic (respectively smooth) coordinate
systems $x$ at the origin, and where $d_x$ is the distance between the Newton polyhedron
and the origin in the coordinates $x.$

A given coordinate system $x$ is said to be adapted to $f$ if $h(f)=d_x.$

Let $\pi$ be the principal face. We assume that $\pi$ is a point or a compact edge, then $f_{\pi}$ is a weighted homogeneous polynomial. Denote by $\nu$ the maximal order of roots of $f_{\pi}$ on the unit circle at the origin, so $$\nu:=\max\limits_{S^{1}}ord(f_{\pi}).$$
If there exists a coordinate system $x$ such that $\nu=d_{x}$ then we set $m=1.$ It can be shown that in this case $x$ is adapted to $f$ (see \cite{IkromovMuller}). Otherwise we take $m=0$. Following A. N. Varchenko we call $m$ the \emph{multiplicity of the Newton polyhedron}.

In the classical paper by A. N. Varchenko \cite{Varch}, he obtained the sharp estimates for oscillatory integrals in  terms of the height.
Also in the paper \cite{IkrActaMath} the height was used to get the sharp bound for maximal operators associated to smooth surfaces in $\mathbb{R}^{3}$. It turns out that analogous quantities can be used for oscillatory integrals with the Mittag-Leffler function.

We consider the following integral with phase $f$ and amplitude $\psi$, of the form
\begin{equation}\label{int1}
I_{\alpha,\beta}=\int_{U} E_{\alpha,\beta}(i\lambda f(x))\psi(x)dx,
\end{equation}
where $0<\alpha<1,$ $\beta>0$, $U$ is a sufficiently small neighborhood of the origin. We are interested in particular in the behavior of $I_{\alpha,\beta}$
when $\lambda$ is large, as for small $\lambda$ the integral is just bounded.
In particular if $\alpha=1$ and $\beta=1$ we have oscillatory integral \eqref{int11}.

  The main result of the work is the following.

\begin{theorem}\label{Theor1}
Let $f$ be a smooth finite type function of two variables defined in a sufficiently small neighborhood of the origin and let $\psi\in C_{0}^{\infty}(U)$.

Let $h$ be the height of the function $f$, and let $m=0,1$ be the multiplicity of its Newton polyhedron.
If $0<\alpha<1$, $\beta>0$, $h>1$, and $\lambda\gg1$ then we have the estimate
\begin{align}\label{Formul5}
& \left|\int_{U}E_{\alpha,\beta}(i\lambda f(x_{1},x_{2}))\psi(x)dx\right|\leq\frac{C|\ln\lambda|^{m}\|\psi\|_{L^{\infty}(\overline{U})}}{\lambda^{\frac{1}{h}}}.
\end{align}

If $0<\alpha<1$, $\beta>0$, $h=1$ and $\lambda\gg1$, then we have following estimate
\begin{align}\label{Formul5000}
& \left|\int_{U}E_{\alpha,\beta}(i\lambda f(x_{1},x_{2}))\psi(x)dx\right|\leq\frac{C|\ln\lambda|^{2}\|\psi\|_{L^{\infty}(\overline{U})}}{\lambda},
\end{align}
where the constants $C$ are independent of the phase,  amplitude and $\lambda$.
\end{theorem}

\section{Auxiliary statements}
We first recall some useful properties.
\begin{proposition}
If $0<\alpha<2,\beta$ is an arbitrary real number, $\mu$ is such that $\pi\alpha/2<\mu<\min\{\pi,\pi\alpha\},$ then there is $C>0,$ such that we have
\begin{equation}\label{Formul1}
  |E_{\alpha,\beta}(z)|\leq\frac{C}{1+|z|}, \,\ z\in\mathbb{C},\,\ \mu\leq|\arg(z)|\leq\pi.
\end{equation}
See \cite{Dzh54a}, \cite{Rudolf}, \cite{Pod}.
\end{proposition}
\begin{proposition}\label{Prop22}
Let $\Omega$ be an open, bounded subset of \,\ $\mathbb{R}^2$, and let $f:\Omega\rightarrow\mathbb{R}$ be a measurable function such that for all $\lambda\gg1$ and for some positive $\delta\neq1$, we have
\begin{equation}\label{Formul1221}
\left|\int_{\Omega}e^{i\lambda f(x)}dx\right|\leq C|\lambda|^{-\delta}|\ln\lambda|^{m},
\end{equation}
with $m\geq0.$ Then, we have
$$\left|x\in\Omega:|f(x)|\leq\varepsilon|\leq C_{\delta}\varepsilon^{\delta}|\ln\varepsilon|^{m},\,\ \text{for}\,\,\, \delta<1,\right.$$
$$\left.\text{for}\,\,\  0<\varepsilon\ll1, \text{and for}\,\,\, \delta>1, |x\in\Omega:|f(x)|\leq\varepsilon|\leq C_{\delta}\varepsilon\right.,$$
$$\text{for}\,\,\, \delta=1, \,\,\ \left|x\in\Omega:|f(x)|\leq\varepsilon|\leq C_{\delta}\varepsilon|\ln\varepsilon|^{m+1},\,\ \right.$$
where  $C_{\delta}$ depends only on $\delta,$ $|A|$ means the Lebesgue measure of a set $A.$ See \cite{JohnGreen}.
\end{proposition}

\begin{proof} For the convenience of the reader we give an independent proof of Proposition  \ref{Prop22}.
We consider an even non-negative smooth function $$\omega(x)=\left\{\begin{array}{l} {1,\, \, \, \, \, \text{when}\, \, |x|\leq 1,} \\ {0,\, \, \, \, \text{when}\, \, |x|\geq 2.} \end{array}\right.$$ For the characteristic function of $\Omega$ with $\overline{\Omega}\subset U$, the following inequality holds true

$$|x\in\Omega:|f(x)|\leq\varepsilon|=\int_{\Omega}\chi_{[0,1]}\left(\frac{|f(x)|}{\varepsilon}\right)dx\leq
\int_{\Omega}\omega\left(\frac{f(x)}{\varepsilon}\right)dx.$$

Now we will use the Fourier inversion formula, and rewrite the last integral as

$$\int_{\Omega}\omega\left(\frac{f(x)}{\varepsilon}\right)dx=
\frac{1}{2\pi}\int_{\Omega}\int_{-\infty}^{\infty}\check{\omega}(\xi)e^{i\xi\frac{f(x)}{\varepsilon}}d\xi dx.$$
As $\check{\omega}(\xi)$ is a Schwartz function, we can use Fubini theorem and change the order of integration. So we have
$$\int_{\Omega}\int_{-\infty}^{\infty}\check{\omega}(\xi)e^{i\xi\frac{f(x)}{\varepsilon}}d\xi dx=
\int_{-\infty}^{\infty}\check{\omega}(\xi)\int_{\Omega}e^{i\xi\frac{f(x)}{\varepsilon}}dxd\xi.$$

We use inequality \eqref{Formul1221} for the inner integral and get
$$\left|\int_{\Omega}e^{i\xi\frac{f(x)}{\varepsilon}}dx\right|\leq\frac{C|\ln(2+\frac{\xi}{\varepsilon})|^{m}}{(1+|\frac{\xi}{\varepsilon}|)^{\delta}}.$$
As $\check{\omega}(\xi)$ is a Schwartz function, we also have
$$|\check{\omega}(\xi)|\leq\frac{C}{1+|\xi|}.$$

So
$$\left|\int_{-\infty}^{\infty}\frac{C\check{\omega}(\xi)|\ln(2+\frac{\xi}{\varepsilon})|^{m}}{(2+|\frac{\xi}{\varepsilon}|)^{\delta}}d\xi\right|\lesssim
\int_{0}^{\infty}\frac{2C|\ln(\frac{\xi}{\varepsilon})|^{m}}{(1+|\xi|)(2+|\frac{\xi}{\varepsilon}|)^{\delta}}d\xi.$$
Now we change the variable as
$\xi=\eta\varepsilon$, and we get
$$\int_{0}^{\infty}\frac{|\ln(\frac{\xi}{\varepsilon})|^{m}}{(1+|\xi|)(2+|\frac{\xi}{\varepsilon}|)^{\delta}}d\xi=
\int_{0}^{\infty}\frac{\varepsilon|\ln\eta|^{m}}{(1+|\varepsilon\eta|)(2+|\eta|)^{\delta}}d\eta.$$
Now we estimate the last integral for different values of $\delta$.

If $\delta<1$ then we have
$$\int_{0}^{\infty}\frac{\varepsilon|\ln\eta|^{m}}{(1+|\varepsilon\eta|)(2+|\eta|)^{\delta}}d\eta\leq
C\varepsilon\int_{0}^{\frac{1}{\varepsilon}}\frac{|\ln\eta|^{m}d\eta}{(2+\eta)^{\delta}}+ C\varepsilon\int^{\infty}_{\frac{1}{\varepsilon}}\frac{|\ln\eta|^{m}d\eta}{\varepsilon\eta^{\delta+1}}.$$
We represent
$\frac{1}{(2+\eta)^{\delta}}=\frac{1}{\eta^{\delta}(1+\frac{2}{\eta})^{\delta}}=\frac{1}{\eta^{\delta}}+O(\frac{1}{\eta^{\delta+1}})$. So
$$C\varepsilon\int_{0}^{\frac{1}{\varepsilon}}\frac{|\ln\eta|^{m}d\eta}{(2+\eta)^{\delta}}=
\varepsilon\int_{0}^{2}\frac{|\ln\eta|^{m}d\eta}{(2+\eta)^{\delta}}+
\varepsilon\int_{2}^{\frac{1}{\varepsilon}}\frac{|\ln\eta|^{m}d\eta}{(2+\eta)^{\delta}}.$$
Integrating by parts we obtain
$$
\varepsilon\int_{2}^{\frac{1}{\varepsilon}}\frac{|\ln\eta|^{m}d\eta}{(2+\eta)^{\delta}}\leq\varepsilon\int_{2}^{\frac{1}{\varepsilon}}
\frac{|\ln\eta|^{m}d\eta}{\eta^{\delta}}\leq C\varepsilon^{\delta}|\ln\varepsilon|^{m}.$$
As $\delta<1$, the integrals
$\int_{0}^{2}\frac{|\ln\eta|^{m}d\eta}{(2+\eta)^{\delta}}$ and $\int^{\infty}_{\frac{1}{\varepsilon}}\frac{|\ln\eta|^{m}d\eta}{\varepsilon\eta^{\delta+1}}$ convergence.

If $\delta>1$ then we trivially obtain
$$\left|\int_{0}^{\infty}\frac{C\varepsilon|\ln\eta|^{m}}{(1+|\varepsilon\eta|)(2+|\eta|)^{\delta}}d\eta\right|\leq
C\varepsilon.$$
If $\delta=1$ then assuming $0<\varepsilon<\frac{1}{2}$ we get $|\varepsilon\eta|<1$ (for $|\eta|<2$), then write the integral as the sum of three integrals and obtain
$$\left|\int_{0}^{\infty}\frac{C\varepsilon|\ln\eta|^{m}}{(1+|\varepsilon\eta|)(1+|\eta|)}d\eta\right|\leq
\left|\int_{0}^{2}C\varepsilon|\ln\eta|^{m}d\eta\right|+$$$$
\left|\int_{2}^{\frac{1}{\varepsilon}}\frac{C\varepsilon|\ln\eta|^{m}}{\eta}d\eta\right|+
\left|\int_{\frac{1}{\varepsilon}}^{\infty}\frac{C\varepsilon|\ln\eta|^{m}}{\eta}d\eta\right|.$$
Then we have
$$\left|\int_{0}^{2}C\varepsilon|\ln\eta|^{m}d\eta\right|\leq C\varepsilon,$$
and we get with simple calculating that
$$\left|\int_{2}^{\frac{1}{\varepsilon}}\frac{C\varepsilon|\ln\eta|^{m}}{\eta}d\eta\right|\leq
C\varepsilon|\ln\varepsilon|^{m+1}.$$
We use the formula of integrating by parts several times, to get
$$\left|\int_{\frac{1}{\varepsilon}}^{\infty}\frac{C\varepsilon|\ln\eta|^{m}}{\eta}d\eta\right|\leq C\varepsilon|\ln\varepsilon|^{m},$$
completing the proof.
\end{proof}
From Proposition  \ref{Prop22} we get the following corollaries.
\begin{corollary}
Let $f(x_{1},x_{2})$ be a smooth function with $f(0,0)=0$, $\nabla f(0,0)=0$, and $h$ be the height of the function $f(x_{1},x_{2})$, and let $m=0,1$ be the multiplicity of its Newton polyhedron. Let also $a(x)=\left\{\begin{array}{l} {1,\, \, \, \, \, \text{when}\, \, |x|\leq \sigma,} \\ {0,\, \, \, \, \text{when}\, \, |x|\geq 2\sigma,} \end{array}\right.$ $\sigma>0,$ and $a(x)\geq0$ with $a\in C_{0}^{\infty}(\mathbb{R}^{2})$. If
for all real $\lambda\gg1$ and for any positive $\delta\neq1$, the following inequality holds
\begin{equation}\label{Formul122}
\left|\int_{\mathbb{R}^{2}}e^{i\lambda f(x)}a(x)dx\right|\leq C|\lambda|^{-\delta}|\ln\lambda|^{m},
\end{equation}
then we have
$$\left||x|\leq\sigma:|f(x)|\leq\varepsilon\right|\leq C\varepsilon^{\delta}|\ln\varepsilon|^{m},$$
where $m\geq0.$ \text{See} \cite{ Greenblat2010, IkromovMuller, IkromovMuller1, PhonStein}.
\end{corollary}

\begin{corollary}\label{corol2}
Let $f(x_{1},x_{2})$ be a smooth function with $f(0,0)=0$, $\nabla f(0,0)=0$, and let $\overline{\Omega}$ be a sufficiently small compact set around the origin. Let also $h$ be the height of the function $f(x_{1},x_{2})$, and let $m=0,1$ be the multiplicity of its Newton polyhedron. Then for all  $0<\varepsilon\ll1$ we have
$$|x\in\Omega:|f(x)|\leq\varepsilon|\leq C\varepsilon^{\frac{1}{h}}|\ln\varepsilon|^{m},$$
where $h$ is the height of $f$ and $m$ is its multiplicity \cite{Greenblat2010}.
\end{corollary}

\section{Proof of the main result}

\textbf{Proof of Theorem \ref{Theor1}.}
As for $\lambda<2$ the integral \eqref{int1} is just bounded, we consider the case $\lambda\geq2$.  Without loss of generality, we can consider the integral over $U.$ Using inequality \eqref{Formul1},  we have
\begin{equation}\label{int10}|E_{\alpha,\beta}(i\lambda f(x))|\leq \frac{C}{1+\lambda|f(x)|}.
\end{equation}

We then use  \eqref{int10} for the integral \eqref{int1}, and  get that
\begin{equation}\label{int16}
|I_{\alpha,\beta}|\leq\left|\int _{U}E_{\alpha,\beta}(i\lambda f(x))\psi(x)dx\right|\leq C\int_{U }\frac{|\psi(x)|dx}{1+\lambda |f(x)|}.
\end{equation}

Now we represent the integral $I_{\alpha,\beta}$ over the union of sets $\Omega_{1}:=\Omega\cap\{\lambda|f(x_{1},x_{2})|<M\}$ and $\Omega_{2}:=\Omega\cap\{\lambda|f(x_{1},x_{2})|\geq M\}$ respectively, where $M$ is a positive real number.

We estimate the integral $I_{\alpha,\beta}$ over the sets $\Omega_{1}$ and $\Omega_{2}$, respectively,
$$|I_{\alpha,\beta}|\leq C\int_{U}\frac{|\psi(x)|dx}{1+\lambda |f(x)|}=J_{1}+J_{2}:=C\int_{\Omega_{1}}\frac{|\psi(x)|dx}{1+\lambda |f(x)|}+C\int_{\Omega_{2}}\frac{|\psi(x)|dx}{1+\lambda |f(x)|}.$$

First we estimate the integral over the set $\Omega_{1}$. Using the results of the paper (\cite{Karpushkin1980} page 31) (see also Corollary \ref{corol2}) we obtain
$$|J_{1}|=C\int_{\Omega_{1}}\frac{|\psi(x)|dx}{1+\lambda |f(x)|}\leq\frac{C|\ln\lambda|^{m}\|\psi\|_{L^{\infty}(\overline{\Omega_{1}})}}{\lambda^{\frac{1}{h}}}.$$

\begin{lemma}\label{lem55}
Let $f\in C^{\infty}$ and $h$ be the height of the function $f$, and let $m=0,1$ be the multiplicity of its Newton polyhedron. For any smooth function $a=a(x,y)$ with sufficiently small support and for $h>1$  the following inequality holds
\begin{equation}\label{int15}I:=\int_{\{|f(x,y)|\geq\frac{M}{\lambda}\}}\frac{a(x,y)}{1+\lambda|f(x,y)|}dxdy\leq \frac{C|\ln\lambda|^{m}\|a\|_{L^{\infty}(\overline{U})}}{\lambda^{\frac{1}{h}}},\end{equation}
where supp$\{a(x,y)\}=U.$
\end{lemma}
\begin{proof}  Let $h>1$.
Consider the sets
$$A_{k}=\left\{x\in U:\frac{2^{k}}{\lambda}\leq|f(x)|\leq\frac{2^{k+1}}
{\lambda}\right\}.$$ For the measure of a set of smaller values we use Lemma $1^{'}$ in the paper \cite{Karpushkin1983} (see also Corollary \ref{corol2}), and we have $$\mu\left(|f(x)|\leq\frac{2^{k+1}}{\lambda},
x\in U\right)
\leq C\left(\frac{2^{k+1}}{\lambda}\right)^{\frac{1}{h}}
\left(\ln\left|\frac{\lambda}{2^{k+1}}\right|\right)^{m}.$$

Let
$$I_{k}:=\int_{A_{k}}\frac{a(x,y)}{1+\lambda|f(x,y)|}dxdy.$$

For the integral
$$\sum\limits_{2^{k}\leq\lambda|f(x)|\leq2^{k+1}} I_{k}=\int_{\Omega_{2}}\frac{a(x,y)}{1+\lambda|f(x,y)|}dxdy,$$ we find the following estimate:
\begin {eqnarray*}|I_{k}|=\left|\int_{A_{k}}\frac{a(x,y)}{1+\lambda|f(x,y)|}dxdy\right|\leq C\|a\|_{L^{\infty}(\overline{U})}\left(\frac{2^{k+1}}{\lambda}\right)^{\frac{1}{h}}
\left|\ln\frac{2^{k+1}}{\lambda}\right|^{m}2^{-k}.
\end {eqnarray*}
From here we find the sum of $I_{k}$ and, by estimating the integral $I$, we get
$$I\leq\|a\|_{L^{\infty}(\overline{U})}\sum_{k=1}^{\infty}I_{k}\leq\|a\|_{L^{\infty}(\overline{U})}\sum_{k=1}^{\infty}\left(\frac{2^{k+1}}{\lambda}\right)^{\frac{1}{h}}
\left|\ln\frac{2^{k+1}}{\lambda}\right|^{m}2^{-k}$$$$\leq\|a\|_{L^{\infty}(\overline{U})}\frac{|\ln\lambda|^{m}}{\lambda^{\frac{1}{h}}}\sum_{k=1}^{\infty}2^{\frac{k+1}{h}-k}k^{m}.$$
As $h>1$,  the last series is convergent, proving the lemma.
\end{proof}
\begin{remark}\label{remark2}
Consider the case $h=1$. The smooth function has non-degenerate critical point at the origin if and only if $h=1$. As $f(x,y)$ is a smooth function with $\nabla f(0,0)=0$, using Morse lemma we have $f\sim x^{2}\pm y^{2}$. So in this case we estimate two sets $\Delta=\Delta_{1}\cup\Delta_{2},$ where $\Delta_{1}:=\{(x,y):\lambda|x^{2}\pm y^{2}|\leq M, |x|\leq1, |y|\leq1\}$ and $\Delta_{2}:=\{(x,y):\lambda|x^{2}\pm y^{2}|> M, |x|\leq1, |y|\leq1\}$.
First we consider the integral over the set $\Delta_{1}$.
Then we have
$$\left|\int_{\Delta_{1}}\frac{a(x,y)}{1+\lambda|x^{2}\pm y^{2}|}dxdy\right|
\leq C\|a\|_{L^{\infty}(\Delta_{1})}\left|\int_{\Delta_{1}}dxdy\right|.$$
Now we estimate the last integral as
$$\left|\int_{\lambda|x^{2}+y^{2}|\leq M}dxdy\right|\leq \frac{C}{\lambda}.$$
Then we estimate the measure of the set $\{|x^{2}-y^{2}|\leq\varepsilon M\},$ where $\varepsilon=\frac{1}{\lambda}$. We have, for simplicity putting $M=1,$
$$\left|\int_{|x^{2}-y^{2}|\leq\varepsilon M}dxdy\right|\leq C\left|\int_{\sqrt{\varepsilon}}^{\sqrt{1-\varepsilon}}
dy\int_{\sqrt{y^{2}-\varepsilon}}^{\sqrt{y^{2}+\varepsilon}}dx\right|=
\left|\int_{\sqrt{\varepsilon}}^{\sqrt{1-\varepsilon}}\left(\sqrt{y^{2}+\varepsilon}-\sqrt{y^{2}-\varepsilon}\right)dy\right|=$$
$$=\left(\frac{y}{2}\sqrt{y^{2}+\varepsilon}+\frac{\varepsilon}{2}\ln|y+\sqrt{y^{2}+\varepsilon}|\right)
\left|_{\sqrt{\varepsilon}}^{\sqrt{1-\varepsilon}}\right.-\left(\frac{y}{2}\sqrt{y^{2}-\varepsilon}-
\frac{\varepsilon}{2}\ln|y+\sqrt{y^{2}-\varepsilon}|\right)
\left|_{\sqrt{\varepsilon}}^{\sqrt{1-\varepsilon}}\right.=$$$$
=\left|\frac{\sqrt{1-\varepsilon}}{2}+\frac{\varepsilon}{2}\ln\frac{\sqrt{1-\varepsilon}+1}{\sqrt{\varepsilon}}-
\frac{\sqrt{2}}{2}\varepsilon-\frac{\varepsilon}{2}\ln|\sqrt{\varepsilon}(1+\sqrt{2})|-\right.$$$$\left.
-\left(\frac{\sqrt{(1-\varepsilon)(1-2\varepsilon)}}{2}-\frac{\varepsilon}{2}\ln|\sqrt{1-\varepsilon}+\sqrt{1-2\varepsilon}|+
\frac{\varepsilon}{2}\ln\sqrt{\varepsilon}|\right)\right|\leq C\varepsilon\ln\varepsilon.$$

Now we consider the integral over the set $\Delta_{2}$.
In this case we change the variables to polar coordinate system and with easy calculating we get \begin{equation}\label{int1115}\left|\int_{\left\{\lambda|x^{2}+y^{2}|\geq M\right\}}\frac{a(x,y)}{1+\lambda|x^{2}+y^{2}|}dxdy\right|\leq
\frac{C|\ln\lambda|\|a\|_{L^{\infty}(\Delta_{2})}}{\lambda}\end{equation}
and
\begin{equation}\label{int12225}\left|\int_{\left\{\lambda|x^{2}-y^{2}|\geq{M}\right\}}\frac{a(x,y)}{1+\lambda|x^{2}-y^{2}|}dxdy\right|\leq
\frac{C|\ln\lambda|^{2}\|a\|_{L^{\infty}(\Delta_{2})}}{\lambda}.\end{equation}
\end{remark}

Now we continue the proof of Theorem \ref{Theor1}. Let $h>1$. We use Proposition \ref{Prop22} for the integral $J_{1}$, to get
$$|J_{1}|\leq\frac{C|\ln\lambda|^{m}\|a\|_{L^{\infty}(\overline{U})}}{\lambda^{\frac{1}{h}}}.$$
Let consider the integral $J_{2}$. If $h>1$, then using Lemma \ref{lem55} we get
$$|J_{2}|\leq\frac{C|\ln\lambda|\|a\|_{L^{\infty}(\overline{U})}}{\lambda^{\frac{1}{h}}}.$$
If $h=1,$ using the Remark \ref{remark2} we get the inequality \eqref{Formul5000}. The proof is complete.

The proof of Theorem \ref{Theor1} shows that if $h=1$, we can get a more precise result.
\begin{proposition}
If $h=1$ and $f$ has an extremal point at the point (0,0) (then $f$ is diffeomorhic equivalent to $x_{1}^{2}+x_{2}^{2}$ or $-x_{1}^{2}-x_{2}^{2}$), then we have
$$|I_{\alpha,\beta}|\leq\frac{C|\ln\lambda|\|\psi\|_{L^{\infty}(\overline{U})}}{\lambda},$$
for all $\lambda\geq2$.
\end{proposition}

\textbf{Declaration of competing interest}

This work does not have any conflicts of interest.

\section*{Acknowledgements} The second author was supported in parts by the FWO Odysseus 1 grant G.0H94.18N: Analysis and Partial Differential Equations and by the Methusalem programme of the Ghent University Special Research Fund (BOF) (Grant number 01M01021) and also supported by EPSRC grant EP/R003025/2.

\subsection*{Data availability} The manuscript has no associated data.


\end{document}